\documentclass[12pt]{amsart}
\usepackage{amsmath,amsthm,amssymb,fullpage}
\usepackage{epsfig}

\newcommand{\size}[1]{\left \vert #1 \right \vert}
\newcommand{\ceil}[1]{\left \lceil #1 \right \rceil}
\newcommand{\floor}[1]{\left \lfloor #1 \right \rfloor}
\newcommand{\prob}[1]{\mathrm{Pr}\left [ #1 \right ]}
\newcommand{\expect}[1]{\mathbb{E} \left [ #1 \right ]}
\newcommand{\eps}{\varepsilon}

\newcommand{\N}{{\mathbb N}}
\DeclareMathOperator*{\tg}{t} 
\DeclareMathOperator*{\disc}{disc}

\newtheorem{theorem}{Theorem}
\newtheorem{lemma}[theorem]{Lemma}

\newtheorem{corollary}[theorem]{Corollary}

\theoremstyle{definition}

\begin{document}
\title{Toppling numbers of complete and random graphs}

\author{Anthony Bonato}
\address{Department of Mathematics, Ryerson University, Toronto, ON, Canada, M5B 2K3}
\email{\tt abonato@ryerson.ca}

\author{William B. Kinnersley}
\address{Department of Mathematics, Ryerson University, Toronto, ON, Canada, M5B 2K3}
\email{\tt wkinners@ryerson.ca}

\author{Pawe{\l} Pra{\l}at}
\address{Department of Mathematics, Ryerson University, Toronto, ON, Canada, M5B 2K3}
\email{\tt pralat@ryerson.ca}

\begin{abstract}
We study a two-person game played on graphs based on the widely studied chip-firing game.  Players Max and Min alternately place chips on the vertices of a graph.  When a vertex accumulates as many chips as its degree, it fires, sending one chip to each neighbour; this may in turn cause other vertices to fire.  The game ends when vertices continue firing forever.  Min seeks to minimize the number of chips played during the game, while Max seeks to maximize it.  When both players play optimally, the length of the game is the {\em toppling number} of a graph $G$, and is denoted by $\tg(G)$.

By considering strategies for both players and investigating the evolution of the game with differential equations, we provide asymptotic bounds on the toppling number of the complete graph. In particular, we prove that for sufficiently large $n$ $$0.596400 n^2 < \tg(K_n) < 0.637152 n^2.$$ Using a fractional version of the game, we couple the toppling numbers of complete graphs and the binomial random graph $G(n,p)$. It is shown that for $pn \ge n^{2/\sqrt{\log n}}$ asymptotically almost surely $t(G(n,p))=(1+o(1))p\tg(K_n)$.
\end{abstract}

\maketitle

\section{Introduction}

The game of chip-firing and its variants have been a subject of active investigation in a variety of disciplines, with applications to topics such as Tutte polynomials, spectral graph theory, matroids, and statistical mechanics; see~\cite{merino} for a survey with an extensive bibliography. A version of chip-firing played on paths was first employed in~\cite{spencer}, and was generalized to arbitrary graphs in~\cite{BLS91}. The so-called Abelian Sandpile model in statistical mechanics~\cite{bak} was introduced independently of the chip-firing game. The idea of the game, played on an grid, is that each vertex is associated a value corresponding to the slope of the sandpile at that site. Once the slope reaches a threshold value, the sandpile collapsed, spreading to adjacent vertices. The Abelian Sandpile model is notable as it is a dynamical system that displays self-organized criticality. Similar processes were studied before; see~\cite{AB1,chipping,brushes1} for definitions, \cite{brushes2, brushes3} for results on random graphs, and~\cite{brushes4} for a combinatorial game.

We consider the following game-theoretic synthesis of the chip-firing game and the Abelian Sandpile model played on undirected (finite) graphs. The game we consider was first suggested by Gregory Puleo to a University of Illinois research group in 2011, and was studied by Cranston and West~\cite{CW13}.  We were inspired by Question 2 on \cite{web_page1} on complete graphs.  We first need a few definitions and observations. A {\em configuration} of a graph $G$ is a placement of chips on the vertices $G$.  We represent a configuration by a function $c: V(G) \rightarrow \N \cup \{0\}$, with $c(v)$ indicating the number of chips on vertex $v$. Let $c$ be a configuration of a graph $G$.  We may {\em fire} a vertex $v$ provided that $c(v) \ge \deg(v)$; when $v$ is fired, $\deg(v)$ chips are removed from $v$, and one chip is added to each of its neighbours.  We call a configuration {\em volatile} if there is some infinite sequence $v_1, v_2, \ldots$ of vertices that may be fired in order.  A {\em stable} configuration is one in which no vertex may be fired; that is, $c(v) < \deg(v)$ for all $v \in V(G)$.
Bj\"orner, Lov\'asz, and Shor~\cite{BLS91} proved that for any configuration of any graph, the order of vertex firings does not matter.  More precisely, if $c$ is a volatile configuration, then after any list of vertex firings, the resulting configuration remains volatile; if instead $c$ is not volatile, then any two maximal lists of vertex firings yield the same stable configuration.  Additionally, they showed that every volatile configuration of a graph $G$ has at least $\size{E(G)}$ chips, and that every configuration having at least $2\size{E(G)}-\size{V(G)}+1$ chips is volatile.

The {\em toppling game} (see \cite{web_page1}) is a two-player game played on a graph $G$.  Initially, there are no chips on any vertices.  The players, Max and Min, alternate turns; on each turn, a player adds one chip to a vertex of his or her choosing.  The game ends when a volatile configuration is reached.  Max aims to maximize the number of chips played before this point, while Min aims to minimize it.  When Max starts and both players play optimally, the length of the game on $G$ is the {\em toppling number} of $G$, denoted $\tg(G)$.  A {\em turn} of the game is the placement of a single chip; a {\em round} is a pair of consecutive turns, one by each player.

Formally, after each turn of the game, vertices are fired until a stable configuration arises (unless of course the configuration is volatile, in which case the game is over).  However, to simplify analysis, we occasionally postpone firing vertices until it is convenient, or stop firing before reaching a stable configuration.  As a result, we may end up firing vertices in a different order than if we had always fired immediately; however, by the result of Bj\"orner, Lov\'asz, and Shor this does not affect whether or not the current configuration is volatile.

Throughout, we consider only finite, simple, undirected graphs in the paper. For background on graph theory, the reader is directed to~\cite{west}.

\subsection{Main results}

We now state our main results and defer the proofs to Sections~\ref{complete} and \ref{random}. We first give asymptotic bounds on the toppling number of complete graphs, which rely in part on the analysis of certain systems of differential equations.

\begin{theorem}\label{main1}
For sufficiently large $n$ we have that $$0.596400n^2 < \tg(K_n) < 0.637152 n^2.$$
\end{theorem}

The proof of Theorem~\ref{main1} is established by proving the upper and lower bounds separately. The main tool is to couple the game with certain auxiliary games (so-called \emph{ideal games}) where we have perfect control of the strategies of both Max and Min. We then derive bounds by simulating the evolution of the ideal games via systems of differential equations (whose numerical solutions yield the constants in Theorem~\ref{main1}).

It may not be evident a priori that the toppling numbers of complete graphs are related to those of random graphs, but our second result shows an intimate connection. For given edge probability $p=p(n)$, we say that a graph property holds {\em asymptotically almost surely} (or {\em a.a.s.}) for $G(n,p)$ if it holds with probability tending to 1 as $n$ tends to $\infty$.

\begin{theorem}\label{main2}
Let $p$ be such that $pn \ge n^{2/\sqrt{\log n}}$, and let $G \in G(n,p)$.  A.a.s.
$$
\tg(G) = (1+o(1))p\tg(K_n).
$$
\end{theorem}
Our approach to prove Theorem~\ref{main2} is to consider a fractional version of the toppling game, and to then make precise the connection between the games (see Theorem~\ref{fractional}). Analysis of the fractional toppling game on $G(n,p)$ uses structural and expansion properties of the random graph (see Lemmas~\ref{lem:expansion} and~\ref{lem:back_edges}).

\section{Complete Graphs}\label{complete}

In this section we establish upper and lower bounds for the toppling number of the complete graph $K_n$ and so prove Theorem~\ref{main1}.  Before proceeding, we establish tools for recognizing volatile configurations (applicable to all graphs).

\begin{lemma}\label{all_vxs_fire}
Let $c$ be a configuration of a graph $G$.  If $c$ admits some list of firings such that every vertex of $G$ fires at least once, then $c$ is volatile.
\end{lemma}
\begin{proof}
To reach a contradiction, suppose that $c$ is not volatile.  By the result of Bj\"orner, Lov\'asz, and Shor~\cite{BLS91}, every maximal list of firings yields the same configuration, so some list of firings both contains all vertices of $G$ and produces a stable configuration $\hat c$.  Index the vertices of $G$ as $v_1, v_2, \ldots, v_n$, in order of their final appearance in this firing list.  Following its last firing, vertex $v_1$ received at least one chip from each of its neighbours; hence $\hat c(v_1) \ge \deg(v_1)$, contradicting stability of $\hat c$.
\end{proof}

In some circumstances it is convenient to compare a given configuration with one that is known to be volatile. For configurations $c_1$ and $c_2$ of a graph $G$, we say that $c_1$ {\em dominates} $c_2$ provided that $c_1(v) \ge c_2(v)$ for all vertices $v$.

\begin{lemma}\label{volatile_domination}
Let $c_1$ and $c_2$ be configurations of a graph $G$ with $c_1$ dominating $c_2$.  If $c_2$ is volatile, then so is $c_1$.
\end{lemma}
\begin{proof}
Let $v$ be a vertex of $G$ for which $c_2(v) \ge \deg(v)$, that is, $v$ may be fired under $c_2$ (and hence also under $c_1$).  Firing $v$ under $c_1$ and under $c_2$ yields new configurations $c'_1$ and $c'_2$, respectively, such that $c'_1$ dominates $c'_2$.  We may repeat this process indefinitely.
\end{proof}

Note also that if $c_1$ is not volatile and $c_1$ dominates $c_2$, then $c_2$ also is not volatile. Next, we have a useful characterization of volatile configurations of $K_n$.

\begin{theorem}\label{volatile_complete}
A configuration $c$ of $K_n$ with $c(v) \le n-1$ for all vertices $v$ is volatile if and only if there is an ordering $v_1, v_2, \ldots, v_n$ of $V(K_n)$ such that $c(v_i) \ge n-i$ for all $i \in [n]$.
\end{theorem}
\begin{proof}
Let $c$ be such a configuration and $v_1, v_2, \ldots, v_n$ such an ordering.  We claim that all $n$ vertices may be fired in order, from which volatility follows by Lemma~\ref{all_vxs_fire}.  Since $c(v_1) \ge n-1$, we may fire $v_1$ immediately.  Once $v_1, v_2, \ldots, v_k$ have been fired for some $k < n$, vertex $v_{k+1}$ has received $k$ additional chips.  Since it had at least $n-k-1$ initially, it now has at least $n-1$, and may be fired.  Repeating as necessary establishes the claim.

For the converse, let $c$ be a volatile configuration satisfying $c(v) \le n-1$ for all vertices $v$.  Let $u_1, u_2, \ldots, u_n$ be a legal firing list under $c$.  We claim that the $u_i$ must be distinct.  Suppose to the contrary that some vertex $v$ appears twice in the sequence: in particular, suppose $u_i = u_j = v$, with $i < j$.  Just before $v$ is fired for the second time, it has received at most $n-2$ chips from earlier firings (since $v$ receives no chip from its first firing).  Moreover, $v$ lost $n-1$ chips on its first firing.  Thus, $v$ has fewer chips at the time of its second firing than it had initially, contradicting $c(v) \le n-1$.  It follows that the $u_i$ are distinct.  For $1 \le i \le n$, vertex $u_i$ can be fired after receiving $i-1$ chips from earlier firings; hence $c(u_i) + i-1 \ge n-1$, which simplifies to $c(u_i) \ge n-i$.  Thus, $u_1, u_2, \ldots, u_n$ is the desired ordering of $V(K_n)$.
\end{proof}

The following result due to Cranston and West~\cite{CW13} will prove useful.  Let the {\em Min-start game} be the variant of the toppling game in which Min, not Max, plays first.
\begin{theorem}[\cite{CW13}]\label{difference}
On any graph, if both players play optimally, then the lengths of the toppling game and Min-start game differ by at most 1.
\end{theorem}

When playing the toppling game on the complete graph, it simplifies analysis to ``sort'' the vertices by the numbers of chips they contain, in nonincreasing order.  At all times we let $v_1$ denote the vertex with the most chips, $v_2$ the vertex with the next most, and so on.  We re-index vertices in this way after every turn of the game.  By the symmetry of the complete graph, we may suppose without loss of generality that, whenever possible, moves are played so that no re-indexing is necessary.  For example, if $v_i$ and $v_{i+1}$ contain equal numbers of chips, then placing a chip on $v_{i+1}$ would force a re-indexing of the vertices: the old $v_i$ becomes the new $v_{i+1}$, and the old $v_{i+1}$ becomes the new $v_i$.  However, we may just as well suppose that the chip was placed directly on $v_i$ to begin with, since this produces an equivalent configuration.  In general, the only moves that force re-indexing are those moves that cause one or more vertices to fire.
\begin{figure}[b!]
\begin{center}
\begin{tabular}{ccc}
\includegraphics[width=1.8in]{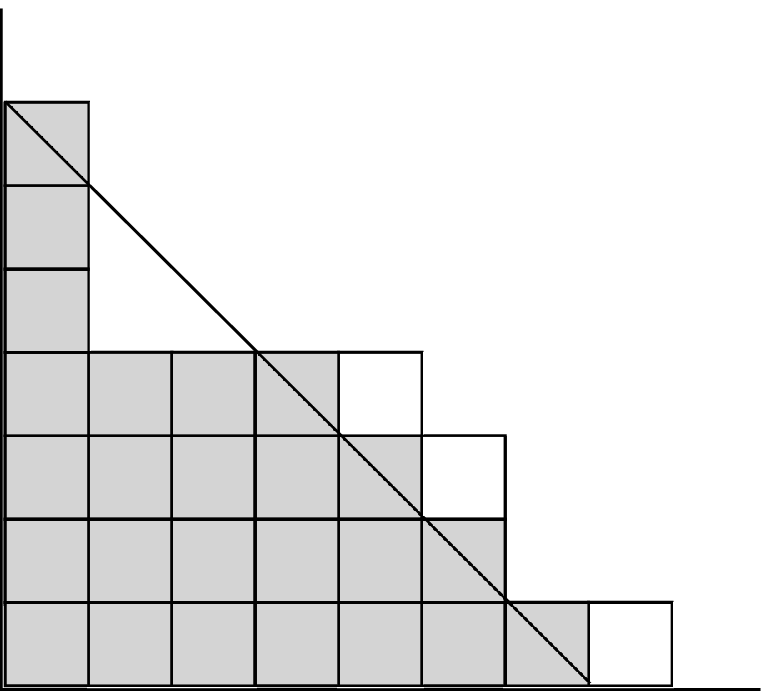} & \hspace{1in} & \includegraphics[width=1.8in]{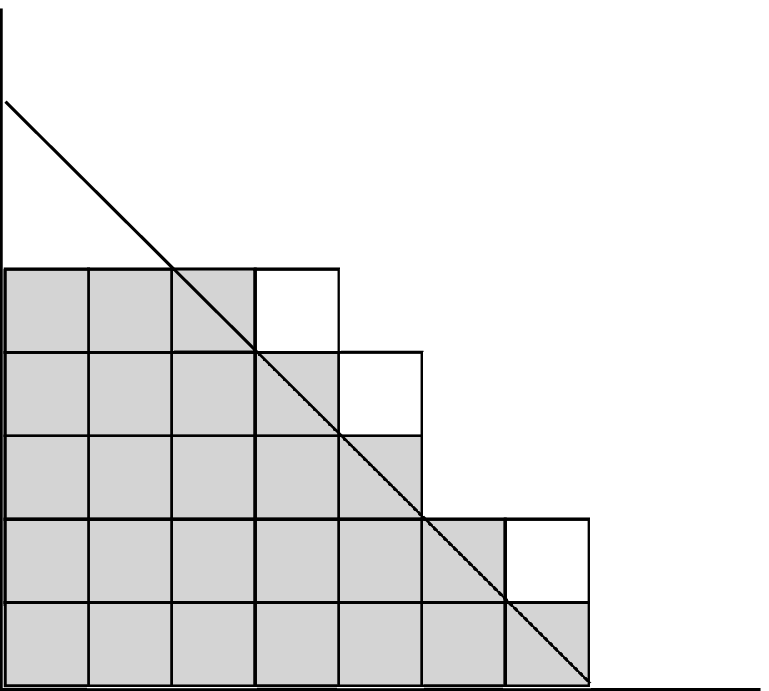}\\
(a) the critical triangle & & (b) after firing \\
\quad\,\,\,\,\, (in-chips are shaded) & &
\end{tabular}
\end{center}
\caption{Firing a single vertex.}\label{fig:triangle}
\end{figure}
It further simplifies matters to represent configurations of $K_n$ graphically, as follows.  Consider a rectangular grid with $n$ columns, with some cells filled and others empty.  Each column represents a vertex, and the number of filled cells in that column indicates the number of chips on that vertex.  A move in the toppling game corresponds to filling a cell in the grid; cells are filled from bottom to top and (in accordance with vertex ``sorting'') from left to right.  For example, cell $(5,8)$ cannot be filled until cells $(4,8)$ and $(5,7)$ have both been filled.  When cell $(1,n-1)$ is filled, vertex $v_1$ fires, and sends one chip to each other vertex in the graph, after which vertices are re-indexed.

By Corollary~\ref{volatile_complete}, the game ends once we reach a configuration $c$ with $c(v_i) \ge n-i$ for all $i \in [n]$.  The set of cells that must be filled to produce a volatile configuration roughly corresponds to a triangle in the grid, which we refer to as the {\em critical triangle}.  We refer to chips placed within the critical triangle as {\em in-chips} and those placed outside as {\em out-chips}.  Since the game ends once $\binom{n}{2}$ in-chips have been played, the length of the game is completely determined by the number of out-chips played.  Thus, Max aims to force many out-chips to be played, while Min aims to prevent this.

In the proof of Theorem~\ref{main1}, we use the observation that firing vertices does not change the numbers of in-chips and out-chips.  When only a single vertex is fired, this is easy to see: firing $v_1$ has the effect of first emptying all cells in the first column, next shifting all filled cells up one unit and left one unit, and finally filling cells $(1,1), (2,1), \ldots, (n-1,1)$---see Figure~\ref{fig:triangle}.  However, when this first firing itself induces more firings, the situation is more complicated.

\begin{lemma}\label{lem_inout}
Let $c$ be a configuration of $K_n$ such that $c(v_1) = n-1$ and $c(v_i) \le n-2$ for $i \ge 2$.  If $c$ is not volatile, then the stable configuration corresponding to $c$ has the same number of in-chips as $c$.
\end{lemma}
\begin{proof}
Suppose $c$ is not volatile, and let $k$ be the number of firings needed to produce a stable configuration.  After each firing, we re-index vertices as needed.  By the result of Bj\"orner, Lov\'asz, and Shor~\cite{BLS91} we may suppose that we only ever fire $v_1$.  For $i \ge 1$, we claim that after $i$ firings, the following three properties hold:
\begin{enumerate}
\item each vertex has at most $n-2+i$ chips,
\item $v_1$ has at most $n$ chips more than $v_n$, and
\item the current configuration has the same number of in-chips as $c$.
\end{enumerate}
When $i=k$, property (3) is precisely the desired claim.  We use induction on $i$.  When $i=1$, property (1) is clear, (2) follows immediately from (1), and (3) follows from the observation preceding the proof.  Suppose now that all three properties hold after $i$ firings, and that $v_1$ is ready to fire.  When we fire $v_1$, each vertex gains at most one chip, so each vertex has at most $n-2+i+1$ chips; hence (1) still holds.  Moreover, $v_n$ gains one chip, while $v_1$ loses $n-1$; hence, property (2) ensures that, after re-indexing, the old $v_1$ becomes the new $v_n$ (and the old $v_2$ becomes the new $v_1$).  Since the old $v_2$ had at most as many chips as the old $v_1$, and the former gained one chip while the latter lost $n-1$, property (2) still holds.  For property (3), suppose that, before the firing, $v_1$ had $n-1+j$ chips.  Since the old $v_1$ becomes the new $v_n$, the graphical effect of the firing is easily described: all cells in the first column are emptied, all remaining cells are shifted up one unit and left one unit, cells $(1,1), (2,1), \ldots, (n-1,1)$ are filled, and cells $(n,1), (n,2), \ldots, (n,j)$ are filled.  Emptying the first column reduces the number of in-chips by $n-1$, shifting cells up and to the left does not change the number of in-chips, filling cells in the first row increases the number of in-chips by $n-1$, and filling cells in the last column again does not change the number of in-chips.  Hence property (3) still holds.
\end{proof}

Before continuing, we outline several possible strategies for Max and Min.
\begin{enumerate}
\item {\bf The ``row strategy'' for Max}: on each turn, Max fills the leftmost empty cell in the bottommost incomplete row.  This has the effect of filling the rows of the grid from left to right, bottom to top.
\item {\bf The ``triangle strategy'' for Min}: Min divides the grid into several ``layers'': layer $i$ consists of those cells whose coordinates sum to $i+1$.  (Note that layers 1 through $n-1$ together comprise the critical triangle.)  On each turn, Min fills the rightmost empty cell in the least-indexed incomplete layer.  This has the effect of filling the layers of the critical triangle in order, from right to left.
\item {\bf The ``square strategy'' for Min}: this strategy is used only as a response to the row strategy for Max.  Let row $k$ be the bottommost incomplete row.  For a nonnegative integer $i$, let $$\quad \quad S_i = \{(1,k+1),(2,k+1),\ldots,(i,k+1),(1,k+2),\ldots,(i,k+2),\ldots,(i,k+i)\};$$ graphically, $S_i$ is the set of cells in the square of side length $i$ with lower-left corner $(1,k+1)$.  Let $s$ be the maximum integer such that all in-chips within $S_s$ have been played; we refer to $S_s$ as ``the square''.  On each turn, Min aims to expand the square by adding another in-chip to $S_{s+1}$.  (When Max completes a row, $k$ increases and the square shifts upward; when this happens, the square temporarily shrinks, until Min can refill the top row.)
\end{enumerate}

We use differential equations to analyze these strategies.  Consider the system of differential equations
\begin{equation}\label{eq:system_x+}
\begin{cases}
y'(x) = \frac {1}{1-z(x)}\\
z'(x) = \frac {1}{z(x)} - \frac {1}{1-z(x)},
\end{cases}
\end{equation}
with initial conditions $y(0)=z(0)=0$. This system arises naturally from the process analyzed in the proof of the next lemma. In particular, it will follow that $y(x)+z(x)$ is an increasing function of $x$. Let $x_+$ denote the value of $x$ for which $y(x)+z(x)=1$. No (simple) closed formula is known for $x_+$; however, we would like to point out that it should be possible to represent it using the \emph{Lambert} $W(x)$ function, which is defined by appropriate solution of $ye^y=x$ and is sometimes called a ``splendid closed formula''~\cite{AMS}. On the other hand, it is straightforward to solve the system of differential equations numerically and derive the following upper bound:
\begin{equation*}\label{eq:x+}
x_+ < 0.318576.
\end{equation*}
We also, independently, verified this numerical value by performing a simulation for large values of $n$. Both the Maple worksheet and computer program used can be downloaded from~\cite{web_page}.

Similarly, for the system of equations
\begin{equation}\label{eq:system1_x-}
\begin{cases}
y_1'(x) = \frac {1}{1-z_1(x)}\\
z_1'(x) = \frac {1}{2z_1(x)} - \frac {1}{2(1-z_1(x))},
\end{cases}
\end{equation}
with initial conditions $y_1(0)=z_1(0)=0$, we define $\bar x$ to be the value of $x$ for which $y_1(x)+2z_1(x)=1$. Numerical solution shows that $\bar x \approx 0.204309$. Finally, consider the system
\begin{equation}\label{eq:system2_x-}
\begin{cases}
y_2'(x) = \frac {1}{1-z_2(x)}\\
z_2'(x) = \frac {1}{2(1-y_2(x)-z_2(x))} - \frac {1}{2(1-z_2(x))},
\end{cases}
\end{equation}
with initial conditions $y_2(\bar x)=y_1(\bar x)$ and $z_2(\bar x)=z_1(\bar x)$. Let $x_-$ denote the value of $x$ for which $y_2(x)+z_2(x)=1$. As before, no closed formula is known for $x_-$, but we have the following bound:
\begin{equation*}\label{eq:x-}
x_- > 0.298200.
\end{equation*}

Now we are ready to analyze the processes that arise under certain choices of strategy.

\begin{lemma}\label{lem:strategies}
Consider the game on $K_n$.
\begin{itemize}
\item [(i)] If Min uses the triangle strategy and Max uses the row strategy, then the game lasts for $(1+o(1)) x_+ n^2$ rounds.
\item [(ii)] If Min uses the square strategy and Max uses the row strategy, then the game lasts for $(1+o(1)) x_- n^2$ rounds.
\end{itemize}
\end{lemma}
\begin{proof}
The approach used to determine the asymptotic behaviour of the processes involved is straightforward, but the details are tedious.  Therefore, we outline the proof only.

We focus on part (i) first. We partition the game into {\em phases}; a given phase ends once Max finishes filling a row. For example, the first phase lasts $n-O(1)$ rounds, and ends when Max fills the last cell in the first row (at which point Min is playing chips in layer $\sqrt{2n}+O(1)$). For a given round $t$, let $f(t)$ be the first incomplete layer and let $g(t)$ be the first incomplete row. To better study the process, we scale the timeline by introducing the function $x = x(t) = t/n^2$. Moreover, we define $y(x) = g(xn^2)/n$ and $z(x)=f(xn^2)/n-y(x)$. Suppose that at the beginning of some phase, the current configuration is described by $y(x)$ and $z(x)$---see Figure~\ref{fig:strategies}(a).

\begin{figure}[htbp]
\begin{center}

\begin{tabular}{ccc}
\includegraphics[width=1.8in]{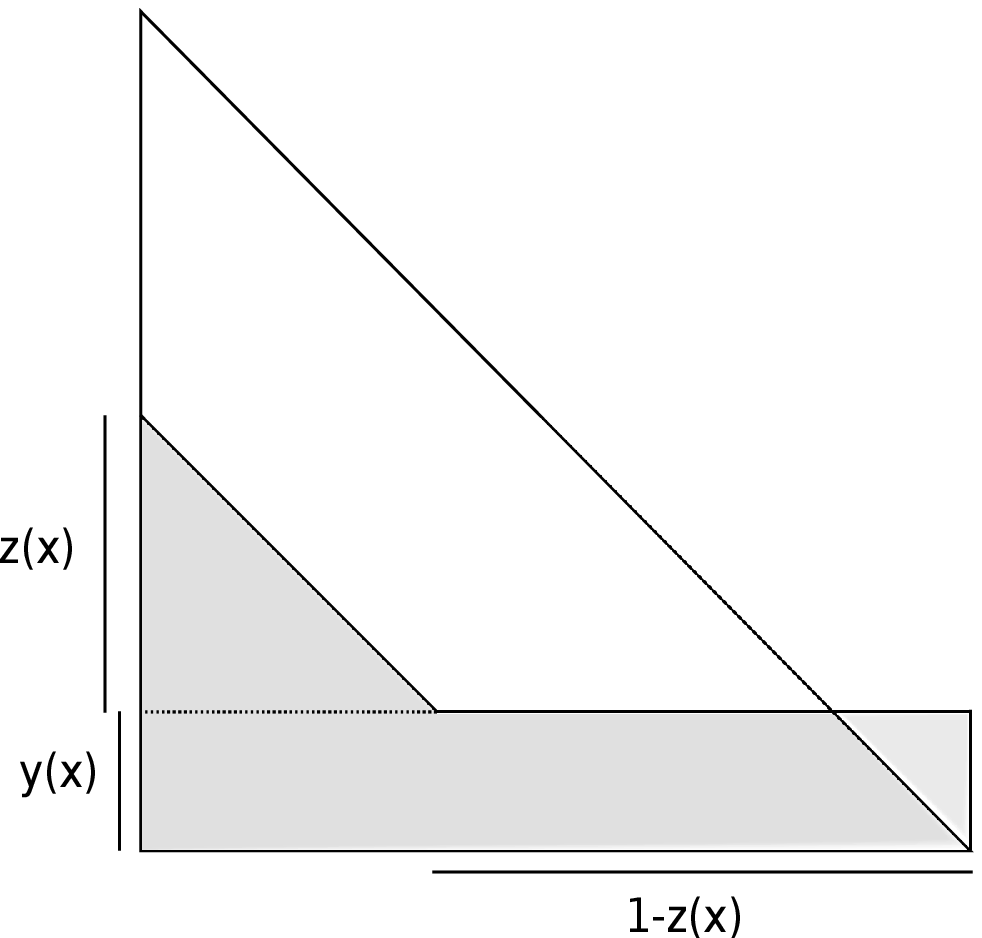} & \includegraphics[width=1.8in]{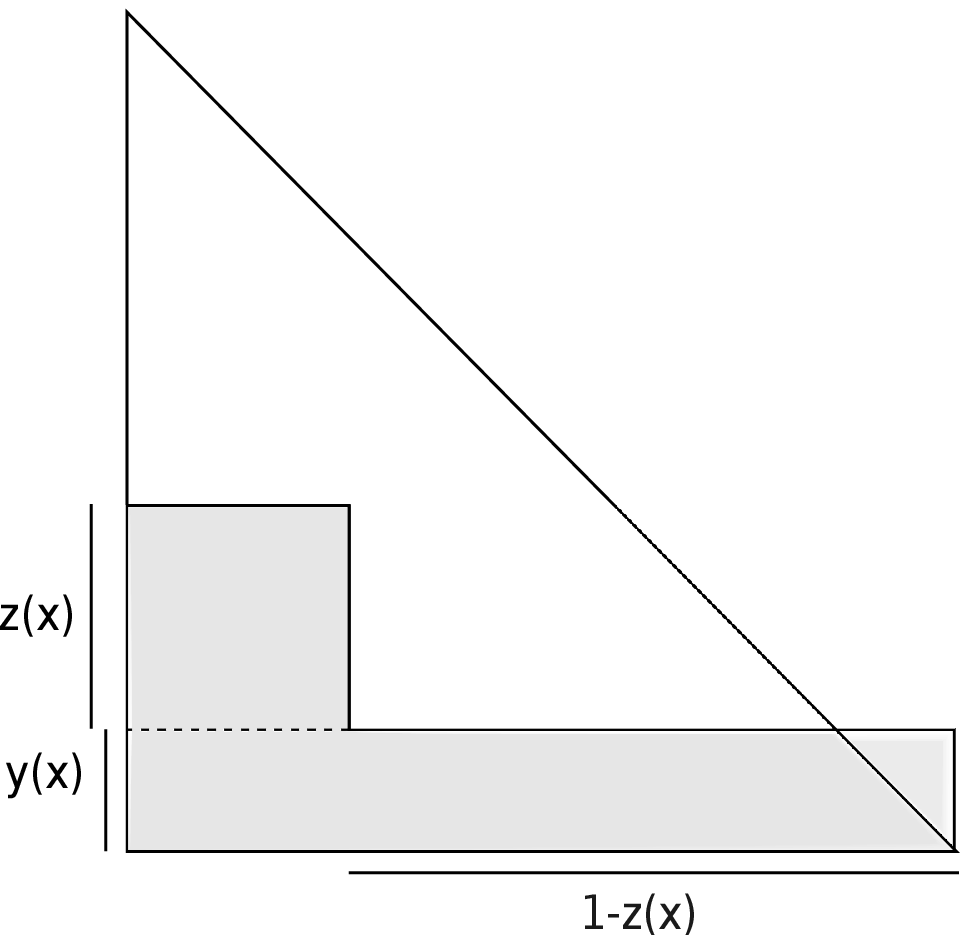} & \includegraphics[width=1.8in]{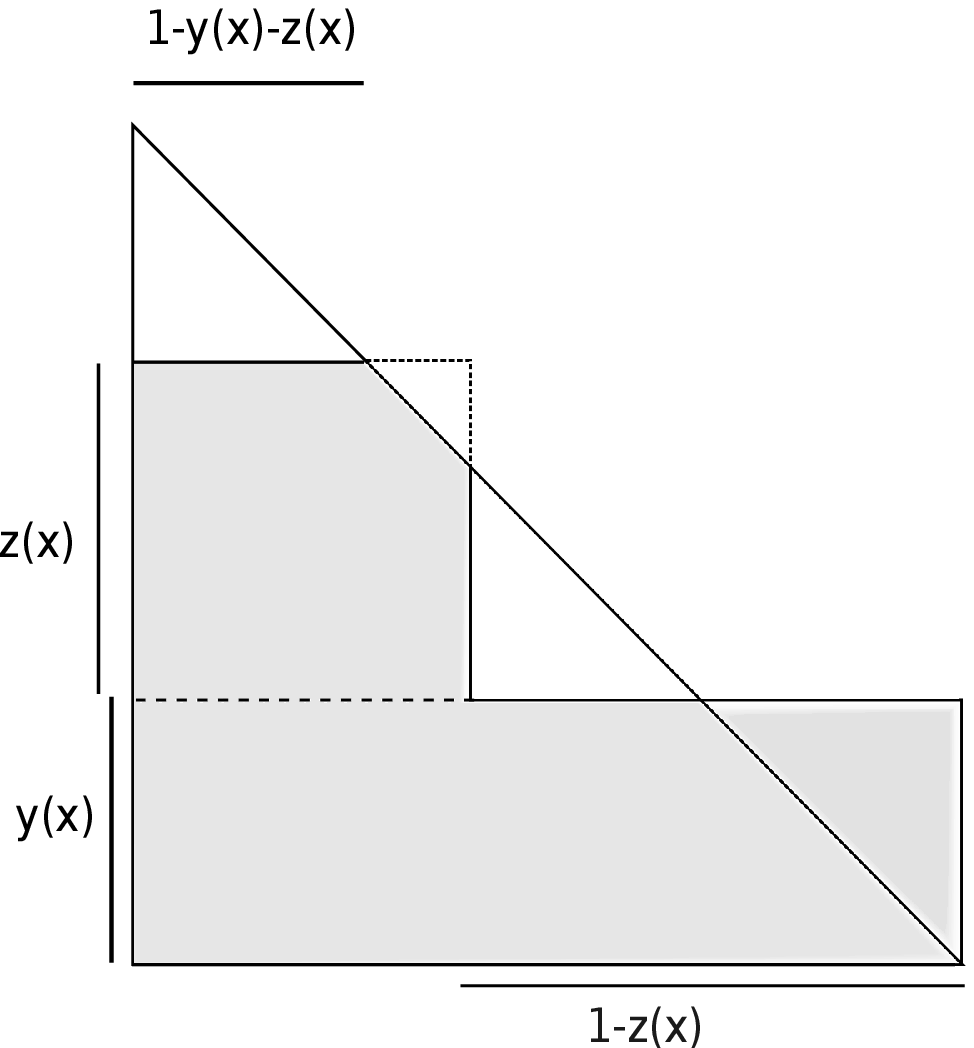}\\
(a) triangle vs.\ row strategy & (b) square vs.\ row strategy & (c) square vs.\ row strategy \\
& (the beginning) & (the end)
\end{tabular}
\end{center}
\caption{Two extreme strategies yielding lower and upper bounds.}\label{fig:strategies}
\end{figure}

The length of this phase is approximately $(1-z(x))n$, and during this phase $y(x)$ increases by $1/n$ (since $g(t)$ increases by 1). Since our goal is to investigate the asymptotic behaviour of this function as $n \to \infty$, we obtain the following differential equation: $y'(x)=1/(1-z(x))$. Similarly, since it takes $z(x)n$ rounds for Min to fill the current layer, during this phase Min fills, on average, $(1-z(x))/z(x)$ layers. Hence, $z(x)$ increases (again, on average) by $((1-z(x))/z(x) - 1)/n$ (note that the current layer is represented by $(y(x)+z(x))n$, and $y(x)$ increases by $1/n$). It follows that $$z'(x) = ((1-z(x))/z(x) - 1)/(1-z(x)),$$ so both functions can be modelled by the system of equations presented in~(\ref{eq:system_x+}). The game ends once the critical triangle has been filled with chips, that is, when $y(x)+z(x)=1$. The proof for this part is finished.

Now we move to part (ii). As before, we partition the game into phases; a given phase ends once Max finishes filling a row. The notation also does not change: for a given round $t$, let $f(t)$ be the side length of the square, and $g(t)$ the first incomplete row.  We focus again on the scaled analogues of $f$ and $g$, namely, the functions $y(x)$ and $z(x)$. Suppose that at the beginning of some phase, the current configuration is described by $y(x)$ and $z(x)$---see Figure~\ref{fig:strategies}(b). The behaviour of the function $y(z)$ is the same as before: the length of this phase is $(1-z(x))n$, and during this phase $y(x)$ increases by $1/n$. Note that once the number of rows is increased, the square ``shifts up'' one row, so Min must stop to fill in the ``missing'' cells of the square.  It takes $z(x)n$ rounds for Min to do this; the remaining $(1-2z(x))n$ rounds are spent expanding the square. Hence, during this phase Min increases the side length of the square by, on average, $(1-2z(x))/2z(x)$ (it takes $2z(x)n$ rounds to increase the length by 1). We obtain that $$z'(x) = ((1-2z(x))/2z(x))/(1-z(x)),$$ so both functions can be modelled by the system of equations presented in~(\ref{eq:system1_x-}).

The behaviour of the function $z(x)$ changes once the square intersects the critical triangle, that is, when $y(x)+2z(x)=1$. In the following phases, Min does not place chips in the whole square but restricts himself to the in-chips alone. Max continues filling rows as before (clearly this is a sub-optimal strategy for Max, since it would be more beneficial for him to play on the square outside the critical triangle)---see Figure~\ref{fig:strategies}(c). At the beginning of each phase, Min fills the ``missing'' cells of the square, which takes $(1-y(x)-z(x))n$ rounds.  In the remaining rounds, Min expands the square; it takes $2(1-y(x)-z(x))n$ rounds to increase the side length by 1. It follows that $$z'(x) = (y(x)/2(1-y(x)-z(x)))/(1-z(x)),$$ which corresponds to the system of equations presented in~(\ref{eq:system2_x-}). The game ends once the critical triangle has been filled, that is, when $y(x)+z(x)=1$. The proof is finished.
\end{proof}

We are now ready to establish upper and lower bounds on $\tg(K_n)$.  Recall that considering the number of chips needed to produce a volatile configuration yields $$\size{E(K_n)} \le \tg(K_n) \le 2\size{E(K_n)}-\size{V(K_n)}+1;$$ Equivalently, we have that $$(0.5+o(1))n^2 \le \tg(K_n) \le (1+o(1))n^2.$$  We asymptotically improve both bounds; note that Theorem~\ref{main1} follows immediately from Theorems~\ref{upper} and~\ref{lower}.

\begin{theorem}\label{upper}
Let $x_+$ be the real number defined by the system of differential equations~(\ref{eq:system_x+}). Then
$$
\tg(K_n) \le (2 x_+ + o(1)) n^2 = (4 x_+ +o(1)) \size{E(K_n)} < 0.637152 n^2.
$$  
\end{theorem}
\begin{proof}
We show that Min can force the claimed upper bound by using the triangle strategy.  By Lemma~\ref{lem:strategies}, it suffices to show that the row strategy for Max is an optimal response to the triangle strategy.

We consider two games.  In the {\em real game}, Min uses the triangle strategy, while Max uses any fixed strategy.  In the {\em ideal game}, Min uses the triangle strategy and Max uses the row strategy.  We claim that the real game finishes no later than the ideal game.  Since it does not affect the asymptotic length of the game, by Theorem~\ref{difference} we may suppose Min plays first.

We divide both games into several {\em phases}.  In either game, for $1 \le i \le n-1$, we let ``phase $i$'' denote the time during which layer $i$ is the first layer that has not yet been completely filled.  Equivalently, this is the period in the game during which Min intends to play in layer $i$.  By an argument similar to that in Lemma~\ref{lem_inout}, when a vertex is fired, the number of chips in each layer remains the same.  In particular, once a layer has been completely filled, it remains so even after vertices are fired.  Thus, once the game has reached phase $i$, it can never return to phase $j$ for any $j < i$.

To prove the claim it suffices to show that, for all $1 \le k \le n-1$, the real game finishes phase $k$ no later than the ideal game.  (Once phase $n-1$ ends, the game ends.)  In both games, phase 1 ends with the first turn, so the claim holds when $k = 1$; we proceed by induction on $k$.

In the ideal game, every phase ends on one of Min's turns; for $1 \le k \le n-1$, let phase $k$ end on Min's $t_k$th turn.  Consider the state of the game just after turn $t_k$, and let $c$ denote the number of empty cells in layer $k+1$ in the ideal game.  At this point, the grid for the ideal game contains chips in all cells of layers 1 through $k$, together with the remainder of the bottommost $k-c$ rows, and perhaps some additional cells in row $k+1-c$.  Clearly Min (not Max) will fill all remaining empty cells in layer $k+1$; hence $t_{k+1} = t_k+c$.  Now consider the state of the real game after round $t_k$.  In total, $2t_k$ chips have been placed.  Since the real game finishes phase $k$ no later than the ideal game, layers 1 through $k$ have all been filled.  Moreover, at least $k+1-c$ rows contain additional chips outside the first $k$ layers.  Hence, layer $k+1$ in the real game has at most $c$ empty cells.  Since Min fills at least one such cell on each of her subsequent turns, phase $k+1$ ends in the real game after at most $c$ additional moves by Min, for at most $t_k+c$ moves in total.  The claimed bound follows.
\end{proof}

\begin{theorem}\label{lower}
Let $x_-$ be the real number defined by the systems of differential equations~(\ref{eq:system1_x-}) and ~(\ref{eq:system2_x-}). Then
$$
\tg(K_n) \ge (2x_-+o(1)) n^2 = (4x_-+o(1)) \size{E(K_n)} > 0.5964 n^2.
$$  
\end{theorem}
\begin{proof}
We give a strategy for Max to force the claimed lower bound.  To do this, we again consider two games, the {\em real game} and the {\em ideal game}.  In the real game, Min uses any fixed strategy, while Max uses the strategy described below.  In the ideal game, Min uses the square strategy and Max uses the row strategy.  By Lemma~\ref{lem:strategies}, it suffices to show that the real game finishes no sooner than the ideal game.  We suppose (by Theorem~\ref{difference}) that Min plays first.

Before presenting Max's strategy, we introduce some further terminology.  At any given point in the game, we say that a row in the grid is {\em complete} if all cells in that row contain chips.  We say that row $i$ is {\em accessible} if all cells in rows $1, 2, \ldots, i$ belonging to the critical triangle contain chips.  (That is, all possible in-chips have been played in the first $i$ rows.)  Note that when any row is accessible but not complete, Max may play an out-chip by playing in the leftmost empty cell in the first such row.  Accessible rows are thus desirable for Max, who aims to play as many out-chips as possible.  It is straightforward to see that firing a vertex increases the number of accessible rows by 1 (and may make some formerly-complete rows incomplete). 

Max plays as follows.  On each turn, if possible, he plays an out-chip in an accessible row; this is possible if and only if some row is both accessible and incomplete.  Otherwise, let row $i$ be the first inaccessible row, and suppose that the first $j$ cells in this row have already been filled.  Let $k$ denote the number of filled cells in column 1.  If $n-i-j < n-1-k$, then Max adds a chip to row $i$; otherwise, he adds a chip to column 1.  (Intuitively, Max aims to create a new accessible row as efficiently as possible, either by filling row $i$ directly or by causing vertex $v_1$ to fire.)

As in the previous proof, we divide the ideal game into {\em phases}.  In the ideal game, no vertices are ever fired (until the final move), so once a row is complete, it remains so for the remainder of the game.  For $i \in \N$, we denote by {\em phase $i$} the period of the game during which row $i$ is the first incomplete row.  In the ideal game, Min never plays an out-chip, so every phase ends after one of Max's turns.  Let $c_i$ denote the number of in-chips played during the first $i$ phases of the ideal game, with $c_0 = 0$.  In the real game, we define {\em phase $i$} to be the period during which the total number of in-chips is at least $c_{i-1}$, but less than $c_i$.

We claim that, for all $k$, at least $\binom{k+1}{2}$ out-chips are played in accessible rows during the first $k$ phases of the real game.  (Note that $\binom{k+1}{2}$ is precisely the number of out-chips played during the first $k$ phases of the ideal game.)  In particular, if both games last for $m$ phases, then at least $\binom{m}{2}$ out-chips must be played in the real game and at most $\binom{m+1}{2}$ in the ideal game; since necessarily $m < n$, the difference between $\binom{m}{2}$ and $\binom{m+1}{2}$ is asymptotically insignificant, so the theorem follows.  The claim holds trivially when $k=1$; we proceed by induction on $k$.

Consider the state of the real game at the beginning of phase $k$.  By the induction hypothesis, at least $\binom{k}{2}$ out-chips have been played in accessible rows.  Suppose first that there are at least $k$ accessible rows.  We aim to show that by the beginning of phase $k+1$, at least $\binom{k+1}{2}$ out-chips will have been played in accessible rows.  If at some point during phase $k$ all accessible rows are complete, then this is clearly the case.  Otherwise Max can, on each of his turns, play an out-chip in an accessible row.  Hence, Max plays at least as many out-chips during phase $k$ of the real game as during phase $k$ of the ideal game, and the claim again follows.

Now suppose instead that at most $k-1$ rows are accessible at the beginning of phase $k$ in the real game.  By assumption at least $\binom{k}{2}$ out-chips have been played in these rows, so there must be exactly $k-1$ accessible rows, all complete.  At this point in the game, let $\ell_r$ denote the number of chips played in column 1 and let $r_r$ denote the number of chips played in row $k$.  Similarly, let $r_i$ denote the numbers of chips played in row $k$ at the beginning of phase $k$ in the ideal game.  In both games, exactly $c_{k-1}$ in-chips have been played and the first $k-1$ rows are complete.  Subject to these constraints, $\min\{n-1-\ell_r,n-k-r_r\}$ is maximized when the remaining in-chips are arranged in a square-like shape, as in the ideal game; in particular, $\min\{n-1-\ell_r,n-k-r_r\} \le n-i-r_i$.  The former quantity is an upper bound on the number of in-chips Max must play before row $k$ becomes accessible, while the latter is the number of in-chips Max plays in the entirety of phase $k$ of the ideal game.  The claim now follows.
\end{proof}

\section{Random Graphs}\label{random}

We now establish a correspondence between the toppling number of $K_n$ and the toppling number of the random graph $G(n,p)$, for $p$ tending to zero sufficiently slowly.  In particular, our main result holds whenever $pn \ge n^{2/\sqrt{\log n}}$ (although most of the lemmas hold for even smaller values of $p$, namely $p \gg \log n / n$).  We begin by introducing a variant of the toppling game that facilitates the connection between the complete graph and the random graph.

Given $p \in (0,1)$, the {\em fractional toppling game} on a graph $G$ is played similarly to the ordinary toppling game, but with different vertex firing rules.  In the fractional game, a vertex $v$ fires once the number of chips on $v$ is at least $p\deg(v)$; when $v$ fires, $p\deg(v)$ chips are removed from $v$, and $p$ chips are added to each neighbour of $v$.  Since $p$ may be any real number, vertices need not contain whole numbers of chips (although each player still places exactly one chip on each turn).  When Max plays first and both players play optimally, the length of the game is denoted by $\tg_p(G)$.

The ordinary toppling game and fractional toppling game are related in a strong sense, made explicit in the following theorem.

\begin{theorem}\label{fractional}
For every $n$-vertex graph $G$, we have that $\tg_p(G) = p\tg(G) + O(n)$.
\end{theorem}
\begin{proof}
We bound $\tg_p(G)$ both above and below in terms of $\tg(G)$.  For the lower bound we give a strategy for Max, and for the upper bound we give a strategy for Min.  Since these strategies are quite similar, we present them simultaneously.  Denote the players by ``A'' and ``B'' (note that A and B could represent either one of Min or Max).  Player A imagines an instance of the ordinary game on $G$ and uses an optimal strategy in that game to guide his or her play in the fractional game.  To simplify analysis, we postpone all firing until the end of the game.

For $t \in \N \cup \{0\}$, let $$r(t) = \max\{s \in \N \cup \{0\} : sp \le t\} = \floor{t/p}.$$  We divide the ordinary game into {\em phases}; phase $t$ consists of $r(t) - r(t-1)$ rounds.  Thus, $t$ rounds in the fractional game correspond to roughly $t/p$ rounds in the ordinary game.  To simplify the analysis, we track the number of chips played at each vertex by each player.  For a vertex $v$, let $x^A_t(v)$ and $x^B_t(v)$ denote the numbers of chips placed on $v$ by A and B, respectively, by the end of phase $t$ in the ordinary game.  Similarly, let $y^A_t(v)$ and $y^B_t(v)$ denote the numbers of chips placed on $v$ by A and B by the end of round $t$ in the fractional game.  We define the {\em discrepancy} at time $t$ for player A by
$$
D^A(t) = \sum_{v \in V(G)} \size{y_t^A(v) - px^A_{t}(v)};
$$
the discrepancy for player B, denoted $D^B(t)$, is defined similarly.

Player A's strategy is as follows.  By Theorem~\ref{difference} we may assume without affecting the asymptotics that B plays first (in both games), so each round of the fractional game consists of a move by B followed by a move by A.  After B plays in round $t$ of the fractional game, A imagines $r(t) - r(t-1)$ moves by B in the ordinary game; he chooses any list of moves minimizing $D^B(t)$.  (This can be done without foreknowledge of A's intervening moves, since said moves do not affect $D^B(t)$.)  Player A responds to each imagined move in turn, according to some optimal strategy for the ordinary game.  Finally A plays, in round $t$ of the fractional game, any move minimizing $D^A(t)$.

Suppose the fractional game lasts for $k$ rounds.  We claim that $D^A(k) < 2n+3$.  Suppose otherwise, and let $k_0$ be the greatest integer such that $D^A(k_0) < 2n+2$.  Fix $\ell > k_0$.  During phase $\ell$ of the ordinary game, A places at most $\ceil{1/p}$ chips.  Thus, when it comes time for A to play in the fractional game,
\begin{equation}\label{eqn_disc}
\sum_{v \in V(G)} \size{y^A_{\ell-1}(v) - px^A_{\ell}(v)} \ge D^A(\ell) - p\ceil{1/p} > 2n.
\end{equation}
By definition of $r(\ell)$, we have that $\sum_{v \in V(G)} y^A_{\ell-1}(v) = \ell-1$ and $\ell-1 < \sum_{v \in V(G)} px^A_{\ell}(v) \le \ell$.  Thus, $\sum_{v \in V(G)} \left(y^A_{\ell-1}(v) - px^A_{\ell}(v)\right) < 0$.  This observation, together with~(\ref{eqn_disc}) and the fact that there are $n$ terms in the sums, implies that $y^A_{\ell-1}(v_0) - px^A_{\ell}(v_0) < -1$ for some vertex $v_0$.  Thus, A can, with his next move in the fractional game, reduce the discrepancy by 1 (for example by playing on $v_0$).  Consequently, over the final $k-k_0$ rounds of the game, the gross increase in $D^A$ is at most $p\ceil{(k-k_0)/p} < k-k_0+1$, while the gross decrease is at least $k-k_0$.  Thus,
$$D^A(k) < 2n+2 + (k-k_0+1) - (k-k_0) = 2n+3,$$
as claimed.  Similarly, $D^B(k) < 2n+3$.  (In fact the bound can be tightened in this case, since B's imagined moves in each phase of the ordinary game can be spread across several vertices, thus offering greater flexibility in decreasing the discrepancy.  However, the stated bound suffices for our purposes.)

We bound the length of the fractional game by showing that both games must finish at roughly the same time.  Toward this end, we make a quick observation.  Let $c$ and $c_f$ be configurations for the ordinary game and fractional game, respectively.  If $pc(v) \ge c_f(v)$ for all vertices $v$, and $c_f$ is volatile for the fractional game, then $c$ is volatile for the ordinary game: this follows because any vertex that may fire in the fractional game may also fire in the ordinary game, and firing vertices preserves the needed inequality.  Similarly, if $c_f(v) \ge pc(v)$ for all vertices $v$, and $c$ is volatile for the ordinary game, then $c_f$ is volatile for the fractional game.

Now suppose that A is in fact Min.  The ordinary game lasts for at most $\tg(G)$ turns, since Min plays optimally.  Since it does not affect the asymptotics, we suppose for convenience that the ordinary game ends after exactly $k$ phases.  At this point, at most $p\tg(G) + O(1)$ turns have elapsed in the fractional game.  If the fractional game has already ended, then Min has enforced the desired upper bound on $\tg_p(G)$, and we are done.  Otherwise, we must show that the fractional game does not last ``too much'' longer.  Let $c$ and $c_f$ be the current configurations for the ordinary game and fractional game, respectively; that is, $c(v) = x^A_k(v)+x^B_k(v)$ and $c_f(v) = y^A_k(v)+y^B_k(v)$ for all vertices $v$.  We have that
$$\sum_{v \in V(G)} \ceil{\size{pc(v) - c_f(v)}} \le D^A(k) + D^B(k) + n \le 5n+6,$$
so after at most $5n+6$ more rounds in the fractional game, Min can produce a configuration $\hat c_f$ such that $\hat c_f(v) \ge pc(v)$ for all vertices $v$.  Since $c$ is volatile for the ordinary game, $\hat c_f$ is volatile for the fractional game, which establishes the desired upper bound on $\tg_p(G)$.

Suppose instead that player A is Max.  Now the ordinary game lasts for at least $\tg(G)$ turns; suppose it ends during phase $k$.  If after round $k$ the fractional game has not yet ended, then it has lasted for at least $p\tg(G) - O(1)$ turns, as desired.  Suppose instead that the fractional game ends during round $k_0$, where $k_0 < k$.  Let $c$ and $c_f$ be the corresponding configurations for the ordinary game and fractional game, respectively; that is, $c(v) = x^A_{k_0}(v)+x^B_{k_0}(v)$ and $c_f(v) = y^A_{k_0}(v)+y^B_{k_0}(v)$ for all vertices $v$.  Now
$$\sum_{v \in V(G)} \ceil{\size{c_f(v) - pc(v)}} \le D^A(k_0) + D^B(k_0) + n \le 5n+6,$$
so after at most $(5n+6)/p$ more rounds in the ordinary game, Min can produce a configuration $\hat c$ such that $p\hat c(v) \ge c_f(v)$ for all vertices $v$.  Since $c_f$ is volatile for the fractional game, $\hat c$ is volatile for the ordinary game.  This implies $\tg(G) \le 2k_0/p + 2(5n+6)/p$, so $p\tg(G) \le 2k_0 + O(n) = \tg_p(G) + O(n)$, as claimed.
\end{proof}

As a special case of Theorem~\ref{fractional}, we obtain the following important corollary.

\begin{corollary}\label{fractional_complete}
If $p \gg 1/n$, then $\tg_p(K_n) = (1+o(1))p\tg(K_n) = \Theta(pn^2)$.
\end{corollary}

We now turn to the toppling game on random graphs.  The central idea behind our main result is that the toppling game on $G(n,p)$ behaves quite similarly to the fractional game on $K_n$ (so long as $p$ tends to zero slowly enough).  The formal proof of this fact is quite lengthy, so we present several lemmas before attacking the main result.  We begin by describing a useful change to the rules of the fractional game that does not asymptotically affect the length of the game on $K_n$.

\begin{lemma}\label{fractional_restricted}
Fix any function $\omega$ tending to infinity with $n$.  Consider the fractional game on $K_n$.  If we forbid either player from placing more than $\omega pn$ chips on the same vertex, then the length of the game (assuming both players play optimally) is $(1+o(1))\tg_p(K_n)$.
\end{lemma}

\begin{proof}
We bound the length of the game above and below.  For the upper bound it suffices to restrict only Min, since restricting Max cannot increase the length of the game.  Likewise, for the lower bound it suffices to restrict only Max.  Both arguments proceed similarly.

Denote the players by ``A'' and ``B''.  We consider two games: the {\em ordinary game}, in which there are no restrictions on moves, and the {\em restricted game}, in which A can play no more than $\omega pn$ chips on any one vertex.  (When restricting A, we consider only chips placed by A, and ignore chips placed by B.)  The restricted game is, in a sense, the ``real'' game, while the ordinary game is ``imagined'' by A as a tool to guide his play in the restricted game.  We give a strategy for A to ensure that the length of the restricted game is asymptotically the same as that of the ordinary game.  To simplify analysis, we postpone all vertex firing until the end of the game.

Player A plays as follows.  At the beginning of the game, A chooses an arbitrary indexing $v_1, v_2, \ldots, v_n$ of the vertices of $K_n$.  Whenever B plays in the restricted game, A imagines that B played identically in the ordinary game.  On A's turn, he first chooses his move in the ordinary game according to some optimal strategy for that game; suppose he chose to play at vertex $v$.  If A has placed fewer than $\floor{(\omega - 1) pn}$ chips on $v$ in the restricted game, then A places a chip on $v$.  Otherwise, A places a chip on the least-indexed vertex on which he has not already placed $\floor{(\omega - 1) pn}$ chips.  (There must be some such vertex given sufficiently large $n$, since A places fewer than $pn^2$ chips before the game ends, and for large $n$ we have that $n\floor{(\omega - 1) pn} \ge pn^2$.)

Suppose now that A is Min.  Min plays optimally in the ordinary game, so it finishes after at most $\tg_p(K_n)$ turns; the restricted game may last longer.  Once the ordinary game finishes, let $c$ and $c_r$ be configurations of the ordinary and restricted games, respectively.  Let $S = \{v \in V(G) : c(v) \not = c_r(v)\}$.  Each vertex in $S$, except perhaps for one, contains $\floor{(\omega - 1) pn}$ chips in the restricted game; hence $\size{S} \le \frac{pn^2}{\floor{(\omega-1)pn}} = o(n)$.  Allow the restricted game to continue for another $\size{S}pn$ rounds; on these extra turns, Min places another $pn$ chips on each vertex of $S$.  Let $\hat c_r$ be the resulting configuration of the restricted game.

We claim $\hat c_r$ is volatile.  Let $u_1, u_2, \ldots$ be an infinite firing sequence for $c$; we construct an analogous firing sequence for $\hat c_r$.  At all times we maintain the invariants that
$$
\sum_{v \in S} (\hat c_r(v) - c(v)) \ge pn\size{S}
$$
and that $\hat c_r(v) \ge c(v)$ for $v \not \in S$.  Both invariants clearly hold before any vertices are fired, and are maintained under the firing (in both games) of any vertex not in $S$.  Suppose we fire, in the original game, some vertex $v$ in $S$.  If $v$ may be fired in the restricted game, then doing so preserves both invariants.  Otherwise, since $\sum_{v \in S} (\hat c_r(v) - c(v)) \ge pn\size{S}$, there is some vertex $v'$ in $S$ such that $\hat c_r(v') - c(v') \ge pn$.  Since $\hat c_r(v') \ge p(n-1)$ we may fire $v'$ in the restricted game, and doing so maintains both invariants.  We may repeat the process indefinitely, hence $\hat c_r$ is volatile.  Consequently, the total number of chips played in the restricted game is at most $\tg_p(K_n) + 2pn\size{S} = \tg_p(K_n) + o(pn^2) = (1+o(1))\tg_p(K_n)$, as desired.

Now suppose instead that A is Max.  By Max's strategy, the ordinary game lasts for at least $\tg_p(K_n)$ turns, but the restricted game may finish earlier.  By an argument similar to that in the preceding paragraph, once the restricted game ends, Min can reach a volatile configuration for the ordinary game after another $o(pn^2)$ turns.  Letting $t_r$ denote the length of the restricted game, we have that $\tg_p(K_n) \le t_r + o(pn^2)$, so $t_r \ge \tg_p(K_n) - o(pn^2) = (1+o(1))\tg_p(K_n),$
as claimed.
\end{proof}

We are now ready to tackle random graphs.  We begin by establishing some properties of $G(n,p)$.  The following well-known result, known as the Chernoff Bound, is very useful (see for example Theorem~2.8~\cite{JLR}).

\begin{theorem}[\cite{JLR}]\label{thm:Chernoff}
Let $X$ be a random variable that can be expressed as a sum $X=\sum_{i=1}^n X_i$ of independent random indicator variables $X_i$, where $X_i \in {\rm Be}(p_i)$ (the $p_i$ need not be equal).  For $t \ge 0$,
\begin{eqnarray*}
\prob {X \ge \expect{X} + t} &\le& \exp \left( - \frac {t^2}{2(\expect{X}+t/3)} \right) \quad \text{ and}\\
\prob {X \le \expect{X} - t} &\le& \exp \left( - \frac {t^2}{2\expect{X}} \right).
\end{eqnarray*}
In particular, if $\eps \le 3/2$, then
\begin{eqnarray*}
\prob {|X - \expect{X}| \ge \eps \expect{X}} &\le& 2 \exp \left( - \frac {\eps^2 \expect{X}}{3} \right).
\end{eqnarray*}
\end{theorem}

We also need the following well-known observation (see e.g.~\cite{BMP} Lemma~2.2). For a given vertex $v$ and integer $i$, let $N(v,i)$ denote the set of vertices at distance $i$ from $v$.

\begin{lemma}\label{lem:expansion}
Let $d = p(n-1)$ and suppose that $\log n \ll d \ll n$.  For $G \in G(n,p)$, a.a.s.\ for every $v \in V(G)$ and $i \in \N$ such that $d^i = o(n)$ we have that
$$
|N(v,i)| = (1+o(1)) d^i.
$$
In particular, a.a.s.\ for every $v \in V(G)$, we have that $\deg(v) = (1+o(1))d$.
\end{lemma}

Our next lemma is an analogue of Lemma~\ref{fractional_restricted} for the ordinary game on $G(n,p)$.

\begin{lemma}\label{random_restricted}
Fix any function $\omega$ tending to $\infty$ with $n$.  Consider the toppling game on $G \in G(n,p)$, where $p \gg \log n / n$.  If we forbid either or both of the players from placing more than $\omega pn$ chips on the same vertex, then a.a.s.\ the length of the game (assuming both players play optimally) is $(1+o(1))\tg(G)$.
\end{lemma}
\begin{proof}
Since we aim to show that the specified bound on $\tg(G)$ holds a.a.s., we may assume that the property stated in Lemma~\ref{lem:expansion} holds deterministically for $G$.

As in the proof of Lemma~\ref{fractional_restricted}, we bound the length of the game above and below; for the upper bound it suffices to restrict only Min, and for the lower bound it suffices to restrict only Max.  The two players' strategies are very similar.

Denote the players by ``A'' and ``B''.  We consider two games: the {\em restricted game}, in which A can play no more than $\omega pn$ chips on any one vertex, and the {\em ordinary game}, in which there are no restrictions on moves.  Player A uses the ordinary game to guide his play in the restricted game.  We give a strategy for A to ensure that the length of the restricted game is asymptotically the same as that of the real game.  To simplify analysis, we postpone firing vertices until it is convenient.

When B makes a move in the ordinary game, A simply imagines that B made the same move in the restricted game.  When A himself plays, more care is needed.  Call a vertex {\em saturated} if A has already placed $(\omega-3) pn$ chips there in the restricted game.   Player A chooses his move for the ordinary game according to some optimal strategy for that game; suppose he plays at vertex $v$.  If $v$ is not saturated, then A plays at $v$ in the restricted game as well.  If in fact $v$ is saturated, then A instead attempts to play (in the restricted game) on some vertex $u$ in $N(v)$.  This may require recursion: if $u$ is saturated, then A attempts to play on some neighbour of $u$, and so on until he reaches an unsaturated vertex. On subsequent moves by A at $v$ in the ordinary game, A chooses different neighbours on which to play in the restricted game, until all neighbours have been used.  At this point we fire $v$ in the ordinary game, which resolves the discrepancy between the two games due to placement of excess chips at $v$.  Player A handles subsequent moves at $v$ similarly---by playing instead at each neighbour of $v$ in turn, firing $v$, and repeating the process.

Suppose A is Min.  Since Min plays optimally in the ordinary game, that game finishes in at most $\tg(G)$ turns.  If at that point the restricted game has already finished, then Min has enforced the desired upper bound, and we are done.  Otherwise, it suffices to show that Min can cause the restricted game to end in $o(pn^2)$ additional rounds, since $\tg(G) \ge \size{E(G)} = \Theta(pn^2)$.  Let $c$ and $c_r$ be the current configurations for the ordinary and restricted games, respectively.  We show that adding $o(pn^2)$ chips to $c_r$ yields a configuration that dominates $c$; since $c$ is volatile, the claim then follows by Lemma~\ref{volatile_domination}.  However, this is straightforward.  There is only one possible reason for discrepancy between $c$ and $c_r$: for each saturated vertex $v$ in the restricted game, it may be that Min attempted to play at some, but not all, neighbours of $v$.  To resolve this discrepancy, it suffices to add, to each such vertex, $\deg(v)$ chips, and then to fire.  Since the game lasts for $\Theta(pn^2)$ turns, the number of saturated vertices is $o(n)$; by Lemma~\ref{lem:expansion}, the number of chips added is thus $o(pn^2)$.  Moreover, no saturated vertex $v$ gains more than $2\deg(v)$ chips during this process, so in total Min plays no more than $(\omega-3)pn + 2\deg(v)$ chips on $v$; for sufficiently large $n$, this is less than $\omega pn$.

Suppose instead that A is Max.  This time, the ordinary game finishes in at least $\tg(G)$ turns.  If at this point the restricted game has not yet finished, then Max has enforced the desired lower bound.  Suppose instead that the restricted game finishes first.  We claim that once the restricted game ends, Min can cause the ordinary game to end within $o(pn^2)$ additional rounds.  Min can do this using a strategy similar to that in the last paragraph: to each saturated vertex $v$, add $\deg(v)$ chips in the ordinary game, and then fire.  This produces a configuration in the ordinary game that dominates the configuration of the restricted game, and hence is volatile.  The desired lower bound follows.
\end{proof}

Before proving the main result, we need two technical lemmas about the structure of the random graph.  The first is straightforward but, lacking a suitable reference, we provide a proof.

\begin{lemma}\label{lem:back_edges}
Let $d = p(n-1)$ and suppose that $n^{2/\sqrt{\log n}} \le d \ll n$.  Fix $G \in G(n,p)$, let $m$ be the largest natural number such that $d^m = o(n)$, and let $\psi = \psi(n) = n / d^m$. A.a.s.\ for every $v \in V(G)$ the following properties hold:
\begin{itemize}
\item [(i)] Fix $j < m$.  For every $u \in N(v,j-1) \cup N(v,j)$, we have that
$$
|N(v,j-1) \cap N(u)| < 1.5\sqrt{\log n}/(m-j).
$$
\item [(ii)] For every $u \in N(v,m-1) \cup N(v,m)$, we have that
$$
|N(v,m-1) \cap N(u)| = O(\log n).
$$
\item [(iii)] For every $u \in V \setminus \bigcup_{j \le m-1} N(v,j)$, we have that
$$
|N(v,m) \cap N(u)| =
\begin{cases}
O(d/\psi) & \text{ if } \psi \le d/\log n, \\
O(\log n) & \text{ if } \psi > d / \log n.
\end{cases}
$$
\end{itemize}
\end{lemma}

\begin{proof}
Since we aim for a result that holds a.a.s., we may assume that the property stated in Lemma~\ref{lem:expansion} holds deterministically for $G$.  Fix vertices $v$ and $u$.

For (i), we perform a breadth-first search from $v$ until the vertices of $N(v,j-1)$ are discovered.  (However, we do not expose any edges from $N(v,j-1)$ to undiscovered vertices, nor do we expose any edges joining vertices in $N(v,j-1)$.)  It follows from Lemma~\ref{lem:expansion} that $|N(v,j-1)|=(1+o(1)) d^{j-1}$. For a fixed vertex $u$, either undiscovered or in $N(v,j-1)$, let $X_u$ be a random variable denoting $|N(v,j-1) \cap N(u)|$. It is straightforward to see that
\begin{eqnarray*}
\prob{X_u \ge k} &\le &\binom{\size{N(v,j-1)}}{k} p^k \\
&\le & \frac {(1+o(1))^k}{k!} d^{k(j-1)}\frac{d^k}{n^k} \\
&\le & \frac{d^{kj}}{n^k} = o\left (\frac{d^{kj}}{d^{km}}\right ) = o(d^{k(j-m)}).
\end{eqnarray*}
Taking $k = 3\frac{\log n}{(m-j)\log d}$ yields that $$\prob{X_u \ge k} = o(d^{-3 \log n / \log d}) = o(n^{-3}).$$  Since there are only $O(n^3)$ choices for $v,u,$ and $j$, the claim follows by the Union Bound.  Moreover, since $d \ge n^{2/\sqrt{\log n}}$, we have that $k \le 1.5\sqrt{\log n}/(m-j)$, as desired.

We prove part (ii) similarly.  Perform a breadth-first search from $v$ until the vertices of $N(v,m-1)$ are discovered (but do not expose any edges from $N(v,m-1)$ to undiscovered vertices, nor any edges joining vertices in $N(v,m-1)$.) It follows from Lemma~\ref{lem:expansion} that $|N(v,m-1)|=(1+o(1)) d^{m-1}$. For a fixed vertex $u$, either undiscovered or in $N(v,m-1)$, let $X_u$ be a random variable denoting $|N(v,m-1) \cap N(u)|$. Now
$$
\expect{X_u} = (1+o(1)) d^{m-1} p = (1+o(1)) \frac {n}{d\psi} p = (1+o(1))\frac {1}{\psi} = o(1).
$$
Applying the Chernoff Bound with $t = 3 \log n$ (see Theorem~\ref{thm:Chernoff}), we find that $X_u = O(\log n)$ with probability $1-o(n^{-2})$. Since there are $O(n^2)$ possibilities for $v$ and $u$, the result holds by the Union Bound.

Finally, consider part (iii).  This time, we stop our breadth-first search once the vertices of $N(v,m)$ have been discovered. As before, for a fixed vertex $u$, either undiscovered or in $N(v,m)$, let $X_u$ be a random variable denoting $|N(v,m) \cap N(u)|$.  Now Lemma~\ref{lem:expansion} yields $|N(v,m)|=(1+o(1)) d^m$, so
\begin{eqnarray*}
\expect{X_u}& = &(1+o(1)) d^m p  \\ &= &(1+o(1)) \frac {n}{\psi} p \\ &= &(1+o(1))\frac {d}{\psi}.
\end{eqnarray*}
If $d/\psi = \Omega(\log n)$, then it follows from the Chernoff Bound that $X_u = O(d/\psi)$ with probability $1-o(n^{-2})$. Otherwise, the Chernoff Bound shows only that $X_u = O(\log n)$ with probability $1-o(n^{-2})$. Item (iii) now follows by the Union Bound.
\end{proof}

We return now to the toppling game on the random graph.

\begin{lemma}\label{firing_limit}
Let $p$ be such that $pn \ge n^{2/\sqrt{\log n}}$.  Fix $G \in G(n,p)$ and let $\omega = \log\log n$.  A.a.s. for every configuration $c$ of $G$ such that $c(v) \le 2\sqrt{\omega}pn$ for all vertices $v$, and every legal firing sequence $F = (u_1, u_2, \ldots, u_n)$ under $c$, every vertex of $G$ appears in $F$ only $O(pn / \log^2 n)$ times.
\end{lemma}
\begin{proof}
We may suppose that the property in Lemma~\ref{lem:expansion} holds deterministically for $G$.  Let $d = p(n-1)$, let $m$ denote the largest integer such that $d^m = o(n)$, and let $\psi = n/d^m$.  If $d = \Omega(n)$ then the bound holds easily: each vertex initially has enough chips to fire only $(2+o(1))\sqrt{\omega}$ times, and receives at most $n-1$ chips from earlier firings, which is enough for only a constant number of additional firings.  Hence, we may suppose that $d = o(n)$, and that the properties in Lemma~\ref{lem:back_edges} also hold deterministically for $G$.

Fix a vertex $v$ in $G$.  For $0 \le i \le m+1$, let $K_i$ denote the number of times vertices in $N(v,i)$ appear in $F$, and let $K = K_0$.  We aim to show that $K = O(pn / \log^2 n)$; suppose to the contrary that $K \gg pn / \log^2 n$.

We first claim that $K_i \ge \left (\prod_{j=1}^i \frac{m-j}{2\sqrt{\log n}} \right ) d^i K$ for $0 \le i \le m-1$.  This is trivially true when $i = 0$.  Fix $i \in \{1, 2, \ldots, m-1\}$, and suppose that $K_{i-1} \ge \left (\prod_{j=1}^{i-1} \frac{m-j}{2\sqrt{\log n}} \right ) d^{i-1} K$.  By Lemma~\ref{lem:back_edges}(i), each vertex in $N(v,i-1)$ has at least $(1-o(1))d$ neighbours in $N(v,i)$.  Thus, with each firing in $N(v,i-1)$, at least $(1-o(1))d$ chips are sent to $N(v,i)$.  In total, the firings in $N(v,i-1)$ send at least $(1-o(1))dK_{i-1}$ chips to $N(v,i)$.

In order to fuel the firings at $N(v,i-1)$, some chips must be received from $N(v,i)$.  By Lemma~\ref{lem:expansion}, at most $(2+o(1))\sqrt{\omega}d^i$ chips begin within distance $i-1$ of $v$, so the number of chips entering $N(v,i-1)$ from $N(v,i)$ must be at least
$$
(1-o(1))dK_{i-1} - (2+o(1))\sqrt{\omega}d^i = (1-o(1))dK_{i-1},
$$
since $\sqrt{\omega} = \sqrt{\log \log n} = o(K)$. By Lemma~\ref{lem:back_edges}(i) each vertex in $N(v,i)$ has at most $1.5\sqrt{\log n}/(m-i)$ neighbours in $N(v,i-1)$, so we have that
\begin{eqnarray*}
K_i &\ge &\frac{m-i}{1.5\sqrt{\log n}} (1-o(1))dK_{i-1} \\
&\ge& \left (\prod_{j=1}^i \frac{m-j}{2\sqrt{\log n}} \right ) d^i K.
\end{eqnarray*}

When $i=m$, it need no longer be the case that each vertex in $N(v,m)$ has so few neighbours in $N(v,m-1)$.  Applying Lemma~\ref{lem:back_edges}(ii) yields
\begin{eqnarray*}
K_m &\ge& \Omega \left( \frac {1}{\log n} \right) (1-o(1)) d K_{m-1} \\
&= &\Omega \left( \left (\prod_{j=1}^{m-1} \frac{m-j}{2\sqrt{\log n}} \right ) \frac {d^m}{\log n}K \right).
\end{eqnarray*}

Finally, consider $K_{m+1}$.  Lemmas~\ref{lem:expansion} and \ref{lem:back_edges}(ii) and (iii) together imply that every vertex in $N(v,m)$ has $(1-o(1))d$ neighbours in $N(v,m+1)$.  By Lemma~\ref{lem:back_edges}(iii), every vertex in $N(v,m+1)$ has at most $O(d\log n /\psi)$ neighbours in $N(v,m)$, so
\begin{align*}
K_{m+1} \ge \Omega\left(\frac{\psi}{d \log n}\right) (1-o(1))dK_{m} &= \Omega \left( \frac{\psi}{d \log n} \cdot \left (\prod_{j=1}^{m-1} \frac{m-j}{2\sqrt{\log n}} \right ) \frac {d^{m+1}}{\log n}K\right )\\
     &= \Omega \left( \frac{n}{\log^2 n} \left (\prod_{j=1}^{m-1} \frac{m-j}{2\sqrt{\log n}} \right ) K \right)\\
     &= \Omega \left( \frac{n}{\log^2 n}K \cdot \frac{(m-1)!}{(2\sqrt{\log n})^{m-1}}\right )\\
     &\ge \Omega \left(\frac{n}{\log^2 n}K \left ( \frac{m-1}{2e\sqrt{\log n}} \right )^{m-1} \right ).
\end{align*}
Since $pn \ge n^{2 / \sqrt{\log n}}$, we have that $m \le 0.5\sqrt{\log n}$.  Moreover, $f(x) = (x / (2e\sqrt{\log n}))^x$ is an increasing function for $x \in [1, 0.5\sqrt{\log n}+1]$.  Consequently,
\begin{eqnarray*}
K_{m+1} &\ge &\Omega \left(\frac{n}{\log^2 n}K \left ( \frac{\sqrt{\log n}}{4e\sqrt{\log n}} \right )^{0.5\sqrt{\log n}} \right ) \\
&\ge &\Omega \left(\frac{n}{\log^2 n}K \left ( \frac{1}{12} \right )^{0.5\sqrt{\log n}} \right ).
\end{eqnarray*}
By assumption $K \gg pn / \log^2 n$, so
$$
K_{m+1} \gg n \cdot \frac{pn}{\log^4 n \cdot 12^{0.5\sqrt{\log n}}}.
$$
However,
\begin{eqnarray*}
pn &\ge &n^{2 / \sqrt{\log n}} \\
&=& \exp(2 \sqrt{\log n}) \\
&\gg &\log^4 n \cdot 12^{0.5\sqrt{\log n}},
\end{eqnarray*}
so $K_{m+1} \gg n$, which is impossible as $F$ has length $n$.  Thus, we have reached a contradiction, from which it follows that $K = O(pn / \log^2 n)$.
\end{proof}

We are almost ready to prove our main result; we need just one more structural lemma about the random graph.

\begin{lemma}\label{discrepancy}
Let $p$ be such that $pn \ge n^{2/\sqrt{\log n}}$.  Fix $G \in G(n,p)$ and let $\omega = \log \log n$.  A.a.s.\ for all sets $S$ such that $\size{S} \le 2n/\omega$, we have either
$$\sum_{v \in V(G)} \size{\disc(v)} = o(pn)\size{S} \quad \text{ or } \quad \sum_{v \in V(G)} \size{\disc(v)} = O(n \log n),$$
where $\disc(v) = \size{N(v) \cap S} - p\size{S}$.
\end{lemma}
\begin{proof}
We may assume that the properties in Lemmas~\ref{lem:expansion} and \ref{firing_limit} hold deterministically for $G$.  In particular, $\sum_{v \in V(G)} \size{\disc(v)} \le \sum_{v \in S} \deg(v) = (1+o(1))pn\size{S}$, so the claim is trivial when $\size{S} < \log n / p$.

Fix a subset $S$ of $V(G)$ with $\log n / p \le \size{S} \le 2n/\omega$.  To establish the claimed bound on $\sum_{v \in V(G)} \size{\disc(v)}$, it suffices to bound $\sum_{v \in V(G) \setminus S} \size{\disc(v)}$ and $\sum_{v \in S} \size{\disc(v)}$ separately; we begin with the former.  Let $\mu = p\size{S}$, and fix $v \not \in S$.  Since $\size{N(v) \cap S}$ is binomially distributed with expectation $\mu$, the Chernoff Bound yields
$$\prob{\size{\disc(v)} \ge \frac{1}{\omega}\mu} \le 2\exp\left (-\frac{1}{3\omega^2}\mu \right ).$$
Let $E$ be the event that there are at least $6\omega^3n/\mu$ vertices $v$ outside $S$ having $\size{\disc(v)} \ge \frac{1}{\omega}\mu$.  Since this property is determined independently for all choices of $v$, we have that
\begin{eqnarray*}
\prob{E} &\le &\binom{n}{6\omega^3n/\mu} 2^{6\omega^3n/\mu}\exp(-2n) \\
&\le& 2^n 2^{o(n)} \exp(-2n) = o(e^{-n}).
\end{eqnarray*}

Similarly, fix $i \in \{2, 3, \ldots, \ceil{\log_2 n}\}$.  The Chernoff Bound yields
\begin{eqnarray*}
\prob{\disc(v) \ge 2^i \mu} &\le &\exp\left(-\frac{(2^i - 1)^2\mu^2}{2\mu(1 + (2^{i} - 1)/3)}\right) \\
& \le &\exp(-2^{i-1}\mu).
\end{eqnarray*}
Now let $E_i$ be the event that there are at least $k_i$ vertices $v$ having $\disc(v) \ge 2^i\mu$, where $k_i = 3n / (2^{i-1} \mu)$.  We have that
\begin{eqnarray*}
\prob{E_i} &\le& \binom{n}{k_i}\exp(-2^{i-1} \mu k_i) \\
& \le& 2^n 2^{-3n} = 2^{-2n} = o(e^{-n}).
\end{eqnarray*}

Thus, the probability that at least one of the events $E$ or $E_i$ holds for $S$ is $o(\log n \, e^{-n})$.  Since there are fewer than $2^n$ choices for $S$, a.a.s.\ for all such sets $S$, neither $E$ nor any of the $E_i$ hold.  But now, for every such set $S$,
\begin{align*}
\sum_{v \in V(G) \setminus S} \size{\disc(s)} &\le n \cdot \frac{1}{\omega}\mu + \frac{6\omega^3n}{\mu} \cdot 4\mu + \sum_{i=1}^{\ceil{\log_2 n}} \left ( 2^{i+1}\mu k_i\right )\\
                                    &= O\left (\frac{pn}{\omega}\right )\size{S} + O(\omega^3 n) + O(n \log n)\\
                                    &= o(pn)\size{S} + O(n \log n),
\end{align*}
as desired.

Finally, consider $\sum_{v \in S} \size{\disc(v)}$.  Letting $m$ denote the random variable counting the number of edges with both endpoints in $S$, we have that
\begin{eqnarray*}
\sum_{v \in S} \size{\disc(v)} &\le &\sum_{v \in S} \size{N(v) \cap S} + \sum_{v \in S} p\size{S} \\
& = &2m + p\size{S}^2 = 2m + o(pn)\size{S}.
\end{eqnarray*}
Hence, it suffices to show that a.a.s. $m = o(pn)\size{S}$.  The random variable $m$ is binomially distributed with expectation $\mu = p\binom{\size{S}}{2} = o(pn)\size{S}$.  By the Chernoff Bound, we have that
$$\prob{m \ge 2\mu} \le \exp\left (-\frac{1}{3}\mu\right ) = o(e^{-n}).$$
Since there are fewer than $2^n$ choices for $S$, a.a.s. we have that $m \le 2\mu$ regardless of choice of $S$.
\end{proof}

We are now finally ready to prove our main result on random graphs.

\begin{proof}[Proof of Theorem~\ref{main2}]
Since we aim for a property that holds a.a.s., we may assume that the properties in Lemmas~\ref{lem:expansion}, \ref{firing_limit}, and \ref{discrepancy} hold deterministically for $G$.  Let $\omega = \log \log n$.

We show that $\tg(G) = (1+o(1))\tg_p(K_n)$, from which the result follows by Theorem~\ref{fractional}.  Once again denote the players by ``A'' and ``B''; we give a strategy for player A.  We play two games, the ordinary game on $G$ and the fractional game on $K_n$.  The ordinary game is the ``real'' game on which both players play, while the fractional game is ``imagined'' by A to guide his strategy for the ordinary game.  Except where otherwise specified, we postpone firing vertices until the end of the game.

Player A plays as follows.  On A's turns, he first plays according to some optimal strategy for the fractional game, and then makes the same move in the ordinary game.  When B plays in the ordinary game, A simply duplicates this move in the fractional game.  Since $\sqrt{\omega} \rightarrow \infty$, by Lemma~\ref{fractional_restricted} we may assume (without affecting the asymptotic length of the game) that A plays no more than $\sqrt{\omega} pn$ chips on any one vertex in the fractional game (and hence also in the ordinary game).  Likewise, we may assume by Lemma~\ref{random_restricted} that B plays no more than $\sqrt{\omega} pn$ chips on any one vertex in the ordinary game (and hence also in the fractional game).  Thus, all vertices have at most $2\sqrt{\omega} pn$ chips at all times in both games.

Suppose first that A is Min.  We aim to show that $\tg(G) \le (1+o(1))p\tg(K_n) + o(pn^2)$ (which suffices since $\tg(K_n) = \Theta(n^2)$).  By Lemmas~\ref{fractional} and \ref{fractional_restricted}, the fractional game ends after at most $(1+o(1))p\tg(K_n)$ turns.  If the ordinary game finishes first, then we are done, so suppose otherwise.  It suffices to show that Min can force the ordinary game to end after $o(pn^2)$ additional rounds or, equivalently, that the ordinary game can be brought to a volatile configuration by adding another $o(pn^2)$ chips.

Let $c$ and $c_f$ denote the configurations of the ordinary and fractional games, respectively.  We would like to fire vertices while maintaining the property that $c(v) \ge c_f(v)$ for all vertices $v$.  (Initially, we have equality for all vertices.)  Since the fractional game has ended, $c_f$ is volatile.  Let $F = (u_1, u_2, \ldots, u_n)$ be a firing sequence of length $n$ for $c_f$ (the $v_i$ need not be distinct).  For $i \ge 1$, define $S_i = \{v : v \text{ appears at least $i$ times in S}\}$.  Since each vertex $u_i$ gains at most $p(n-1)$ chips from predecessors in $S$, we must have $c_f(v) > (i-1)p(n-1)$ whenever $v \in S_i$.  Let $k = \max\{i : S_i \not = \emptyset\}$.  Construct a sequence $F'$ by listing all elements of $S_k$, followed by all elements of $S_{k-1}$, and so on down to $S_1$, with the restriction that when listing the elements of $S_1$ we do so in order of their final appearances in $S$.  We claim that $F'$ is a legal firing sequence for $c_f$: since each chip in $S_i$ can fire $i-1$ times without ``assistance'' from earlier vertices, the only potential problems come from the firings in $S_1$, but each vertex receives at least as many chips before firing as it did under $F$.

We aim to show that $F'$ is also a legal firing sequence in the ordinary game (after adding $o(pn^2)$ more chips).  We do this by firing vertices in large groups.  With the $S_i$ defined as in the previous paragraph, let $\ell = \min\{i : \size{S_i} \ge n/\omega\}$.  Let $F^*$ be the portion of $F'$ consisting of those vertices in $S_k, S_{k-1}, \ldots, S_{\ell+1}$.  In the fractional game, each vertex has at most $2\sqrt{\omega}pn$ chips under $c_f$, and gains at most $p(n-1)$ more as $F'$ is fired; consequently, $k \le (2+o(1))\sqrt{\omega}$. Thus, $\size{F^*} \le (2+o(1))\sqrt{\omega} \cdot \frac{n}{\omega} = o(n)$.  For each $i \in \{1, 2, \ldots, \ell\}$, we divide the portion of $F'$ corresponding to $S_i$ into consecutive blocks with sizes between $n/\omega$ and $2n/\omega$.

We have thus expressed $F'$ in the form $(F^*, F_1, \ldots, F_m)$, where $\size{F^*} = o(n)$, no vertex appears more than once in any $F_i$, and each $F_i$ has size between $n/\omega$ and $2n/\omega$.  We aim to fire each subsequence in both games, possibly after adding a few extra chips in the ordinary game.  At all times we maintain the property that each vertex has at least as many chips in the ordinary game as in the fractional game.  Each time we attempt to fire a vertex $v$ in the ordinary game, we first add $\ceil{\deg(v) - p(n-1)}$ chips to $v$.  This ensures that $v$ loses exactly $p(n-1)$ chips in the course of firing; it also guarantees that whenever we may fire $v$ in the fractional game, we may also fire it in the ordinary game.  By Lemma~\ref{lem:expansion}, we add only $o(pn^2)$ chips across all $n$ firings.

We first fire all of $F^*$.  In the fractional game, each vertex gains at most $p\size{F^*}$ chips from these firings.  This need not be the case in the ordinary game.  To compensate, before firing any vertices in the ordinary game, we add $\ceil{p\size{F^*}}$ chips to each vertex; this ensures that, after firing, each vertex still has at least as many chips in the ordinary game as in the fractional game.  Since $\size{F^*} = o(n)$, we only add $o(pn^2)$ extra chips.

We next fire the $F_i$ in order.  Let $F_i = (v_1, v_2, \ldots, v_m)$, and recall from the construction of $F_i$ that the $v_j$ must be distinct.  When firing $F_i$, some vertices may receive more chips in the fractional game than in the ordinary game.  As before, we compensate for this discrepancy before firing.  To each $v_j$ we add $p\size{F_i}$ chips, to ensure that $v_j$ will have enough chips to fire when it is reached in the firing sequence; the total number of chips added is $p\size{F_i}^2 = o(pn^2/\omega)$.  Each vertex $v$ outside $F_i$ receives exactly $p\size{F_i}$ chips from the firing in the fractional game, but only $\size{F_i \cap N_G(v)}$ chips in the ordinary game.  Thus, to each vertex $v$ we add $\max\{\size{F_i \cap N_G(v)} - p\size{F_i}, 0\}$ chips; by Lemma~\ref{discrepancy}, the number of extra chips needed is only $o(pn^2/\omega)$.  The number of chips added to compensate for the discrepancy due to each $F_i$ is $o(pn^2/\omega)$, so the total number added throughout the full firing sequence is $o(pn^2)$.

Let $U$ be the set of vertices not appearing in $F'$.  All vertices not in $U$ have been fired at least once.  Since we have fired $n$ vertices, each vertex in $U$ has at least $pn$ chips in the fractional game, and hence also in the ordinary game.  By Lemma~\ref{lem:expansion}, we may now add another $o(pn)$ chips to each vertex in $U$ to ready it for firing.  Corollary~\ref{all_vxs_fire} now implies that we have reached a volatile configuration, and it follows that $\tg(G) \le (1+o(1))p\tg(K_n) + o(pn^2)$ as claimed.

Suppose now that A is Max.  We now aim to show that $$\tg(G) \ge (1+o(1))p\tg(K_n) - o(pn^2) = (1+o(1))p\tg(K_n).$$  Play both games until the ordinary game ends; By Lemma~\ref{random_restricted}, the ordinary game lasts for at least $(1+o(1))\tg(G)$ turns.  If the fractional game finishes first, then we are done, so suppose otherwise.  It suffices to show that Min can force the fractional game to end after $o(pn^2)$ additional rounds or, equivalently, that the fractional game can be brought to a volatile configuration by adding another $o(pn^2)$ chips.  This would imply that $(1+o(1))\tg_p(K_n) \le (1+o(1))\tg(G) + o(pn^2)$ which, by Lemma~\ref{fractional}, is equivalent to the desired lower bound on $\tg(G)$.

Let $c$ and $c_f$ be the current configurations of the ordinary and fractional games, respectively.  Initially $c_f(v) = c(v)$ for all vertices $v$; we aim to fire vertices while maintaining $c_f(v) \ge c(v)$ for all $v$.  During this process we may need to add $o(pn^2)$ extra chips to the fractional game.  Let $F = (u_1, u_2, \ldots, u_n)$ be a legal firing sequence under $c$.  Partition $F$ into contiguous subsequences $F_1, F_2, \ldots, F_k$, where each $F_i$ has size between $n/\omega$ and $2n/\omega$, and $k \le \omega$.  We show how to fire the $F_i$, in order, in both games.

Fix $i$; we aim to fire $F_i$ in both games and re-establish $c_f(v) \ge c(v)$ for all $v$, while adding only $o(pn^2/\omega)$ chips to the fractional game.  For sufficiently large $n$, Lemma~\ref{firing_limit} ensures that each vertex of $G$ appears at most $Cpn/\log^2 n$ times in $F_i$, for some constant $C$.  For $1 \le j \le Cpn/\log^2 n$, let $S_j$ be the set of vertices appearing at least $j$ times in $F_i$.  Define $\disc_j(v) = \size{N(v) \cap S_j} - p\size{S_j}$.  When $F_i$ is fired in both games, vertex $v$ receives $\sum_j (\disc_j(v) + p\size{S_j})$ chips in the ordinary game, but only $\sum_j p\size{S_j}$ chips in the fractional game.  To compensate for this discrepancy, before firing $F_i$, we add $\max\{\sum_j \disc_j(v), 0\}$ chips to each vertex $v$ in the fractional game.  By Lemma~\ref{discrepancy}, we have that
\begin{eqnarray*}
\sum_{v} \sum_j \disc_j(v) &\le & \sum_j o(pn)\size{S_j} + \sum_j O(n \log n)  \\
&= & o(pn^2/\omega) + O\left (n \log n \frac{Cpn}{\log^2 n} \right ) \\
&= & o(pn^2/\omega).
\end{eqnarray*}
We must exercise caution when firing, as vertices in $F_i$ may receive chips ``sooner'' in the ordinary game than in the fractional game.  To compensate for this, to each vertex $v$ appearing in $F_i$, we add an additional $p\size{F_i}$ chips; consequently, $v$ receives $\sum_j \size{N(v) \cap S_j}$ chips in the fractional game before any vertices of $F_i$ are fired, which ensures that each vertex in $F_i$ receives in advance, in the fractional game, all chips it receives from firing $F_i$ in the ordinary game.  Thus, when $v$ may be fired in the ordinary game, it may also be fired in the fractional game.  In total, this costs only $O(pn^2/\omega^2) = o(pn^2/\omega)$ extra chips.  Finally, when a vertex is fired in the ordinary game, it may lose $o(pn)$ fewer chips than in the fractional game; we compensate for this by adding extra chips to each vertex in the fractional game before it is fired.  Again, this requires adding only $o(pn^2/\omega)$ extra chips in total.

We now fire all $F_i$ in order.  For each $F_i$ fired, we add $o(pn^2/\omega)$ chips to the fractional game, so in total we add $o(pn^2)$ chips.  Since we have fired $n$ vertices in the fractional game, every vertex has either fired or has received $pn$ chips and hence may now fire.  Thus, by Lemma~\ref{all_vxs_fire}, the fractional game has ended.  Since we have added only $o(pn^2)$ chips to the fractional game, the desired lower bound on $\tg(G)$ follows.
\end{proof}

Throughout the section we have assumed that $pn \ge n^{2/\sqrt{\log n}}$.  However, this assumption is needed only to establish Lemma~\ref{firing_limit} (and is, consequently, used also in Lemma~\ref{discrepancy}).  If this lemma holds under the weaker assumption that $p \gg \log n / n$ (and with the weaker conclusion that every vertex fires only $o(pn)$ times), then the remainder of the proof of Theorem~\ref{main2} could be made to work as well.  In particular, the upper bound, which does not require Lemma~\ref{firing_limit} or the full strength of Lemma~\ref{discrepancy}, can already be established for $p$ in this range.  On the other hand, note that for $p \le \log n / n$, with positive probability $G(n,p)$ contains at least one isolated vertex, in which case the toppling game ends immediately.  An open and seemingly non-trivial problem is to investigate the toppling number of the giant component.

\end{document}